\documentclass[a4paper,12pt,draft]{amsart}
\usepackage{eucal}
\usepackage{amsmath,amsthm,amssymb}
\usepackage{amsfonts}
\usepackage{latexsym}
\usepackage[dvips]{graphics}
\usepackage{amsrefs}
\usepackage{mathrsfs}
\theoremstyle{plain}
\newtheorem{thm}{Theorem}[section]
\newtheorem{prop}[thm]{Proposition}
\newtheorem{lemma}[thm]{Lemma}
\newtheorem{cor}[thm]{Corollary}

\theoremstyle{definition}
\newtheorem{defn}[thm]{Definition}

\theoremstyle{remark}
\newtheorem{rem}[thm]{Remark}

\numberwithin{equation}{section}

\makeatletter
\@namedef{subjclassname@2020}{\textup{2020} Mathematical Subject Classification}
\makeatother
\title{A model structure on the category of equivariant $A$-modules over a Hopf algebra} 
\date{\today} 
\author{Mariko Ohara}
\address{Center for Liberal Arts and Sciences, Faculty of Engineering, Toyama Prefectural University, 5180, Kurokawa, Imizu City, Toyama Prefecture, JAPAN, 939-0398. 
\\ Tel : +81-766-56-7500.
}  
\email{primarydecomposition@gmail.com}
\newcommand{\Ker}{{\mathrm{Ker}}}
\newcommand{\Z}{{\mathbb{Z}}}
\newcommand{\Hom}{\mathrm{Hom}}

\newcommand{\lmod}{\underline{\mathrm{LMod}}}
\newcommand{\LMod}{\mathrm{LMod}}

\usepackage{enumerate}
\usepackage{array}
\usepackage[all]{xy} 
\usepackage{delarray}
\ifx\undefined\bysame
\newcommand{\bysame}{\leavemode\hbox to3em{\hrulefill}\,}
\fi
\begin{document}
\thispagestyle{empty}

\begin{abstract}
Let $H$ be a finite dimensional Hopf algebra over a field $k$ and $A$ an $H$-module algebra over $k$. 
In \cite{Khov} and \cite{Qi}, Khovanov and Qi defined acyclic objects and quasi-isomorphisms by using null-homotopy and contractible objects. They also defined the cofibrant objects, the derived category of $A \# H$-modules, and showed properties of compact generators. On the other hand, a model structure on the category of $A \# H$-modules are not mentioned yet. 

In this paper, we check that the category of $A\# H$-modules admits a model structure, where the cofibrant objects and the derived category are just those defined in \cite{Qi}. We 
show that the model structure is cofibrantly generated. 
\end{abstract}

\maketitle


\section{Introduction}
Let $H$ be a finite dimensional Hopf algebra over a field $k$, which is not semisimple, and $A$ an $H$-module algebra, and we denote by $A \# H$ be a $k$-algebra which is called smash algebra of $A$ and $H$ in this paper. 

Recall that the category of $H$-modules admits a model structure, called the stable model structure, whose cofibrations and fibrations are monomorphisms and surjections, respectively, and weak equivalences are stable equivalences by using projective or injective objects which plays a role of contractible objects. 

The category of modules over a certain Hopf algebra can be regarded as a category of chain complexes or dg modules since there is a category equivalence between the category of chain complexes and the category of $\Z$-graded $\Lambda (d)$-modules. Here, $\Lambda (d)$ be the exterior algebra generated by a degree $-1$ element $d$ regarded as a Hopf algebra. 

Khovanov~\cite{Khov} defined acyclic objects by using projective or injective objects, and show that, for a finite dimentional Hopf algebra $H$ which is not semisimple, the quotient category of $H$-modules with respect to stable equivalences are triangulated. 

Qi~\cite{Qi} developed the work to the category of $A \# H$-modules. In \cite{Qi}, the triangulated category and derived category of $A \# H$-modules are given by using the stable equivalences in the category of $H$-modules. He also defined the cofibrant objects and compact objects in the category of $A \# H$-modules and showed the existence of functorial cofibrant replacement functor and gave the compact generator in the derived category. 
They calculated the Grothendieck groups $K_0$ and $G_0$ of those derived categories.  Qi also implies the existence of higher $K$-theory. 

However, we note that any model structure of the category of $A \# H$-modules is not given yet.  

As in \cite{TO}, if we define weak equivalences as those maps which induce isomorphisms on Hopf-homologies in all degree by using suspension, we obtain a model structure on the category of $A \# H$-modules. However, the derived category associated this model structure is different from the derived category defined in \cite{Qi}.

In this paper, we adopt the stable equivalences as weak equivalences, which is the same as in \cite{Qi}, surjections as fibrations and those maps satisfying left lifting property as cofibrations. It is the same thing using the following adjunction 
$$F : \LMod_H \rightleftarrows \LMod_{A \# H} : U,$$ 
where $F$ is the free functor and $U$ is the forgetful functor. We define fibrations and weak equivalences in $\LMod_{A \# H}$ to be those maps in $\LMod_H$ via the forgetful functor $U$. 
We show that these three classes of maps define a model structure on the category of $A \# H$-modules and see that cofibrant objects in the model category and associated derived category is just defined in \cite{Qi}. 
We also check that the model structure is cofibrantly generated. 

Note that the relative model structure on the category of dg modules over a dg algebra plays the central role of relative homological algebra. The relative model structure is not cofibrantly generated in general, so that it is different from 
 the model structure in this paper. (cf. \cite[Remark 4.7]{BMR}) 

We assume the condition $\varepsilon (\lambda) \ne 0$ for the left integral $\lambda$ on the Hopf algebra $H$ to show the factorization axiom. This assumption implies that we focus on a kind of group-algebra type Hopf algebra. 
The main theorem is the following. 

\begin{thm}
Let $k$ be a field. Let $H$ be a finite-dimensional Hopf algebra which is not semisimple. Assume $\varepsilon (\lambda) \ne 0$ for a left integral $\lambda \in H$ and the counit map $\varepsilon : H \to k$. Let $A$ be a left $H$-module algebra over $k$. 
Then, the category of left $A \# H$-modules admits a model structure given by the following three classes of maps: 
\begin{enumerate}[(i)]
\item The fibrations and weak equivalences are given by those maps which are fibrations and weak equivalences as $H$-module maps, respectively. 
\item The cofibrations are given as those maps which has the left lifting property with respect to each map which is both a fibration and a weak equivalence. 
\end{enumerate}
The model structure is cofibrantly generated. Moreover, the cofibrant objects are precisely those objects defined in \cite{Qi} and the associated derived category is the same as in \cite{Qi}. \end{thm}

\subsubsection*{Acknowledgement}
The author would like to thank Professor Dai Tamaki to suggest study Hopf algebras. The author also would like to thank Professor Takeshi Torii to advise comparing with the stable category of certain spectra.

\section{Stable module category of $H$-modules and $A \# H$-modules}
Let $k$ be a field and $H$ a finite dimensional Hopf algebra over $k$. 
The Hopf algebra $H$ which is finite dimensional over a field $k$ is a Frobenius algebra, so that a projective $H$-module is just an injective $H$-module. We assume that $H$ is not semisimple throughout this paper for the existence of the left integral. Here, we say that $\lambda \in H$ is the left integral if it satisfies $h\lambda = \varepsilon(h)\lambda$ for all $h \in H$. 
We also note that the antipode $S$ of $H$ is bijective. 

We denote by $\LMod_H$ the category of $H$-modules. 
The category $\LMod_H$ inherits a canonical model structure as the category of modules over a Frobenius algebra as follows.

\begin{defn}
A morphism $f, g : X \to Y$ in $\LMod_{H}$ is a stably equivalent if $f-g$ factors through a projective $H$-module. This is an equivalence relation which is compatible with composition. A morphism $f, g : X \to Y$ of $H$-modules is a stable equivalence if it is an isomorphism after taking stable equivalence classes of $H$-module morphisms.  
\end{defn}

\begin{defn}\label{ModelH}
Let $f : X \to Y$ be a morphism in $\LMod_{H}$. 
\begin{enumerate}[(i)]
\item We say that $f$ is a weak equivalence if it is a stable equivalence. 
\item We say that $f$ is a fibration if it is a surjection. 
\item We say that $f$ is a cofibration if it is a monomorphism. 
\end{enumerate}
\end{defn}

All objects are cofibrant and fibrant, so the homotopy category $\lmod_H$ is just obtained as the quotient category with respect to the stably equivalent relation and the derived category, denoted by $\mathcal{D}(\LMod_H)$, that is given by inverting the class of weak equivalences on the subcategory of cofibrant objects with respect to this model structure is just obtained by inverting stable equivalences on the category $\LMod_H$. 

Khovanov~\cite{Khov} showed that the category $\lmod_H$ is triangulated via the following explicit suspension and desuspension. 

\begin{lemma}\label{lem23}
Let $\Sigma (M) = M \otimes (H / (\lambda))$ and $\Sigma^{-1} (M) = H \otimes \Ker (\varepsilon)$. Here, $\varepsilon : H \to k$ is the counit map. Then $\Sigma$ and $\Sigma^{-1}$ induce functors on $\lmod_{H}$ that are inverse each other. 
\end{lemma}
\begin{proof}
The assetion follows from the fact, proved by Khovanov~\cite{Khov}, that there exists a projective $H$-module $Q$ such that
\[
\Ker(\varepsilon) \otimes (H / (\lambda)) \cong k \oplus Q . 
\]
\end{proof}

The homotopy category $\lmod_{H}$ becomes a triangulated category with $\Sigma$ and $\Sigma^{-1}$. 

As in \cite{Qi}, we recall some properties of left $A \# H$-modules and the derived category of left $A \# H$-modules. 

Let $H$ be a finite dimensional Hopf algebra over a field $k$. 
A $k$-algebra $A$ is called a left $H$-module algebra if $A$ is a left $H$-module such that $h(ab)=\Sigma (h_{(1)}a)(h_{(2)}b)$ and $h1_A=\varepsilon(h)1_A$ for all $a, b \in A$ and $h \in H$. Here, we use the Sweedler notation for Hopf algebras and $\varepsilon$ is the counit of $H$. 

The smash product algebra, or semidirect product, of $A$ with $H$, denoted by $A \# H$, is the vector space $A \otimes_k H$, whose elements are denoted by $a \# h$ instead of $a \otimes h$, with multiplication given by $(a \# h)(b \# \ell) = \Sigma a(h_{(1)}b) \# h_{(2)}\ell$ for $a, b \in A$ and $h , \ell \in H$. The unit of $A \# H$ is $1 \# 1$. 
The left $H$-module structure on Hom-set $\Hom_{A \# H}(M, N)$ for left $A \# H$-modules $M$ and $N$ is given by the formula $(hf)(m)= \Sigma h_{(2)}f(S^{-1}(h_{(1)})m)$, where $S$ is the antipode of the Hopf algebra $H$. 

Let $M$ and $N$ be $A \# H$-modules. If a $A$-module map $f : M \to N$ is an $H$-linear, we have $(hf)(m)=h_{(2)}f(S^{-1}(h_{(1)})m)=h_{(2)}S^{-1}(h_{(1)})f(m)=\varepsilon(h)f(m)$, so that $f$ has the trivial $H$-action via the counit $\varepsilon$. 
Conversly, if $f$ has the trivial $H$-action via the counit map, we know $f(hm)=f(\varepsilon (h_{(2)})h_{(1)}m)=\varepsilon(h_{(2)})f(h_{(1)}m)=h_{(2)}f(h_{(1)}m)=h_{(3)}f(S^{-1}(h_{(2)})h_{(1)}m)=h_{(3)}f(\varepsilon(h_{(2)})m)=hf(m)$. Therefore, $f$ is $H$-linear. 

To summarize, we have the following lemma ; 
\begin{lemma}[\cite{Qi}, Lemma 5.2]
Let $M$, $N$ be $A \# H$-modules. Any $f \in \Hom_{A \# H}(M, N)$ can be regarded as an element in $\Hom_A(M, N)$ by setting $hf = \varepsilon (h) f$ for $h \in H$. 

Conversely, any $f \in Hom_A(M, N)$ on which $H$ acts trivially extends to a $A \# H$-module homomorphism. In other words, we have a canonical isomorphism of $k$-vector spaces $\Hom_{A \# H}(M, N)= \Hom_H(k_0, \Hom_A(M, N))$. Here, $k_0$ is the abstract $1$-dimensional trivial $H$-module by $k_0$, i.e., $k_0 \cong k v_0$, where $k v_0$ is a $k$-vector space spanned by $v_0$ and its $H$-action is given by $hv_0 = \varepsilon (h) v_0$ for $h \in H$. 

\end{lemma}

We identified $A \# H$-modules with $H$-equivariant $A$-modules by the following lemma.  (cf. \cite[Proposition 2.12]{TO})

\begin{lemma}\label{equiv}
A morphism of $A \# H$-modules is equivalent to say an $H$-equivariant $A$-module map. 
Furthermore, the category of left $A \# H$-modules is categorycally equivalent to the category of $H$-equivariant $A$-modules.  
\end{lemma}
%
\qed

For a left $H$-module $V$ and a left $A \# H$-module $M$, we have an isomorphism 
$$\Hom_A(A \otimes_k V , M) \cong \Hom_k(V , M)$$ of left $H$-modules. 
Since $H \cong k \otimes_k H$ is a subalgebra of $A \# H$, we have the forgetfull functor $\mathrm{U} : \LMod_{A \# H} \to \LMod_H$ by restricting scalars. 
Note that this forgetfull functor is exact functor. We obtain the free and forgetfull adjunction 
$$A \otimes_k (-) : \LMod_H \rightleftarrows \LMod_{A \# H} : U,$$ 
where $A \otimes_k (-) $ is the free functor and $U$ is the forgetful functor. 
We have a stable equivalence of $A \# H$-modules as follows. 

\begin{defn}\label{AHstable}
A morphism $f, g : X \to Y$ in $\LMod_{A \# H}$ is a stably equivalent if $U(f)-U(g)$ factors through a projective $H$-module. This is an equivalence relation which is compatible with composition. A morphism $f, g : X \to Y$ of $A \# H$-modules is a stable equivalence if it induces an isomorphism after taking stable equivalence classes of $A \# H$-module morphisms.  
\end{defn}

We can also define the associated category $\lmod_{A \# H}$ of $\LMod_{A \# H}$, whose objects are the same as those of $\LMod_{A \# H}$ and the morphism set between two left $A \# H$-modules $X$ and $Y$ is defined by the quotient $\Hom_{\lmod_{A \# H}}(X, Y) = \Hom_{\LMod_{A \# H}}(X, Y) / \mathcal{I}(X, Y)$, where $\mathcal{I}(X, Y)$ is the space of morphisms between $X$ and $Y$ that factor through a left $A \# H$-module which is a projective $H$-module. We will say that those morphisms are null-homotopic, and we will say that a left $A \# H$-module $M$ is contractible if it is a projective as a left $H$-module. 

The shift functor on $\lmod_{A \# H}$, which we also write $\Sigma$ by abuse of language, are also given as follows : for any $A \# H$-module $M$, we have a morphism $M \to M \otimes_k H$ from $M$ to $M \otimes_k H$, given by $\mathrm{id}_M \otimes_k \lambda$. Here, $\lambda$ is a left integral in $H$. 
Then $\Sigma (M)$ is defined to be the pushout $M \otimes_k (H / k \Lambda)$ along this inclusion. 
Similary, the inverse suspension is defined by $\Sigma^{-1}(M)=M \otimes_k \Ker \varepsilon$, where $\varepsilon : H \to k$ is the counit map of $H$.  The stable category $\lmod_{A \# H}$ becomes a triangulated category~\cite{Qi}. 

The forgetfull functor $\mathrm{U} : \LMod_{A \# H} \to \LMod_H$ induces an exact functor
\[
\underline{U} : \lmod_{A \# H} \to \lmod_H
\]
between the two quotient categories, which is compatible with taking quotient categories. 

\begin{rem}\label{Rem}
In \cite{Qi}, a morphism $f : M \to N$ in $\lmod_{A \# H}$ is said to be a quasi-isomorphism if its restriction $\underline{U}(f)$ is an isomorphism in $\lmod_H$. 
Then, quasi-isomorphisms in $\lmod_{A \# H}$ constitute a localizing class~\cite[Theorem 4.2.1]{Qi}. The localization of $\lmod_{A \# H}$ via quasi-isomorphisms is denoted by $\mathcal{D}(A, H)$, and is called the derived category of left $A \# H$-modules. 
Note that morphisms in $\LMod_{A \# H}$ whose equivalence class is in the localizing class is just stable equivalences in Definition~\ref{AHstable}. 
\end{rem}

%
%

Here we have some basic lemmas for studying these triangulated categories. 

\begin{lemma}
For any $H$-module $M$, $H \otimes M \cong M \otimes H$ is a free $H$-module. If $M$ is fimite dimensional over $k$, $H \otimes M$ is a free $H$-module of rank $dim_k M$. 

Especially, for a projective $H$-module $P$, $P \otimes M$ and $M \otimes P$ are projective $H$-modules for any $H$-module $M$. 
\end{lemma}
\qed

\begin{lemma}[\cite{Qi}, Lemma 4.3]
Let $0 \to L \to M \to N \to 0$ be an $A$-split short exact sequence in $\LMod_{A \# H}$. Then, we obtain a distinguished triangle in $\lmod_{A \# H}$ of the form $L \to M \to N \to \Sigma (L)$. 

Conversely, any distinguished triangle in $\lmod_{A \# H}$ is arising from an $A$-split short exact sequence in $\LMod_{A \# H}$. 
\end{lemma}
\qed

\begin{rem}
Let us consider the exact sequence $0 \to X \to X \otimes H \to \Sigma X \to 0$. The inclusion $id_X \otimes \lambda :  X \to X \otimes H$ is $A$-split. We take a pushout of this exact sequence along the arbitrary map $X \to Y$. Then, we obtain the exact sequence $0 \to Y \to C_f \to \Sigma X \to 0$, which is also $A$-split. This exact sequence corresponds to the standard triangle $X \to Y \to C_f \to \Sigma X$. 
\end{rem}


\section{A model structure on the category of $A \# H$-modules}

Now, we use the adjunction
\[
A \otimes_k (-) : \LMod_H \rightleftarrows \LMod_{A \# H} : U , 
\]
where $U$ is the functor which regards $A \# H$-modules as $H$-modules and $A \otimes_k (-)$ is the left adjoint, which is given by the free functor.

By using the model structure on $\LMod_H$ as in Definition~\ref{ModelH} and this adjunction, we take the following three classes of maps in the category $\LMod_{A \# H}$. 

\begin{defn}\label{ModelAH}
Let $f : X \to Y$ be a morphism in $\LMod_{A \# H}$. 
\begin{enumerate}[(i)]
\item We say that $f$ is a weak equivalence if $Uf$ is a weak equivalence. 
\item We say that $f$ is a fibration if $Uf$ is a fibration. 
\item We say that $f$ is a cofibration if $f$ has the left lifting property with respect to trivial fibrations. 
\end{enumerate}
\end{defn}

We will see that the three classes of maps as in Definition~\ref{ModelAH} defines a model structure on the category $\LMod_{A \# H}$. Note that the category $\LMod_{A \# H}$ is bicomplete. 

\begin{lemma}[2 out of 3]
Let $f$, $g$ and $g \circ f$ be morphisms in $\LMod_{A \# H}$. If two of the three  morphisms are weak equivalences in $\LMod_{A \# H}$, then, so is the third. 
\end{lemma}
\begin{proof}
Since $U$ is a covaiant functor, the composition $U(f) \circ U(g)$ is $U(f \circ g)$. The assertions follow from restricting these maps on $H$-modules under the functor $U$.  
\end{proof}
\begin{lemma}
Let $f$ and $g$ be maps of $\LMod_{A \# H}$ such that $f$ is a retract of $g$, i.e., $f$ and $g$ satisfies the following commutative diagram. 
\[ 
 \xymatrix@1{
M \ar[d]_f \ar[r] & C \ar[r] \ar[d]_g & M \ar[d]_f \\
N \ar[r] & D \ar[r] & N 
} \] 
where the horizontal composites are identities. 

If $g$ is a weak equivalence, cofibration, or fibration, respectively, then so is the third. 
\end{lemma}
\begin{proof}
In the case of $g$ is a weak equivalence or fibration, respectively, we can see that a retract $f$ of $g$ is also a weak equivalence or fibration, respectively, by regarding the diagram as a diagram of $H$-modules via the functor $U$.  
If the map $g$ is a cofibration, it suffices to show that for any trivial fibration $X \to Y$, the following diagram admits the left lifting
\[ 
 \xymatrix@1{
M \ar[d]_f \ar[r] & X \ar[d]  \\
N \ar[r] &  Y .
} \] 
Now, we add the commutative square to the left hand side of the above diagram as follows and the map $g$ has the left lifting property
\[ 
 \xymatrix@1{
M \ar[r] \ar[d]_f & C \ar[d]_g \ar[r]& M \ar[d]_f \ar[r] & X \ar[d] \\
N \ar[r] &D \ar[r] \ar@{.>}[urr] & N \ar[r] & Y .
} \] 
By composite the map $N \to D$ with the dotted arrow, we have  the left lifting of $f : M \to N$ with respect to $X \to Y$. 
\end{proof}

The lifting propeties for trivial cofibrations and trivial fibrations are obvious by restricting a certain commutative square of $A \# H$-modules to a commutative square of $H$-modules and by the definition of cofibrations, fibrations and weak equivalences. 

From now on, for a left integral $\lambda \in H$ and the counit map $\varepsilon : H \to k$, we assume that $\varepsilon (\lambda)$ is not equal to zero. This is an assumption on $H$. This assumption is needed to the definition below and to show the factorization axiom.  

We consider an analogue of loop space. 

\begin{prop}
For an $A \# H$-module $M$, we take a set $\Hom_k (H, M)$ of $k$-module maps. The $A$-module structure is defined by $(a \phi)(g)=a\phi(g)$ for $a \in A$, $\phi \in \Hom_k(H, M)$ and $g \in H$.  For $\phi \in \Hom_k(H, M)$, define an action of $h \in H$ on $\phi$ by $(h \cdot \phi )(g)=\phi (S(h) g)$. Then, it defines a left $H$-module structure on $\Hom_k(H, M)$. It is compatible with the action of $A$. 
\end{prop}
\begin{proof}
The $H$-module structure on $\Hom_k(H, M)$ is associative by definition. We see that the $A$-module structure is $H$-equivariant. For $a \in A$, $h, h' \in H$ and $f \in \Hom_k(H, M)$, we have $((h_{(1)}a)(h_{(2)}f))(h')=(h_{(1)}a)f(S(h_{(2)})h')=(\varepsilon(h_{(1)}a))f(S(h_{(2)}h'))=af(S(\varepsilon(h_{(1)})h_{(2)})h')=af(S(h)h')=(h(af))(h')$. 
\end{proof}

We consider analogious objects of mapping cylinder and mapping path space. 
\begin{defn}
Assume that $\varepsilon (\lambda) \ne 0$. We define an $A \# H$-module $C_f$ by the pushout
\[ 
 \xymatrix@1{
X \ar[r]^f \ar[d]_{id_X \otimes \lambda} & Y \ar[d]  \\
X \otimes H \ar[r] & C_f .
} \] 

As the following diagram, we obtain the map $C_f  \to \Sigma X$ by the universality. 
\[ 
 \xymatrix@1{
X \ar[r]^f \ar[d]_{id_X \otimes \lambda} & Y \ar[d] \ar[r] & 0 \ar[dd]  \\
X \otimes H \ar[r] \ar[d]^{id} & C_f  \ar@{.>}[dr] &\\
X \otimes H \ar[rr] & & \Sigma X 
} \] 

By virtue of the assumption that $\varepsilon (\lambda) \ne 0$, we have the commutative diagram
\[ 
 \xymatrix@1{
X \ar[r]^f \ar[d]_{id_X \otimes \lambda} & Y \ar[d] \ar[rdd]^{id_Y} & \\
X \otimes H \ar[r] \ar[drr]_{f \otimes \varepsilon} & C_f  \ar@{.>}[dr] &\\
 & & Y \otimes k \cong Y .  
} \] 

Similarly, we define an $A \# H$-module $P_f$ by the pullback

\[ 
 \xymatrix@1{
P_f \ar[r] \ar[d]& X \ar[d]^f  \\
\Hom_k(H, Y) \ar[r]_{\lambda^{\ast} \otimes id_Y} & Y \otimes_k k \cong Y.
} \] 

Since $H$ is finite dimensional over $k$, the linear dual $H^{\ast}=\Hom_k(H, k)$ of $H$ is isomorphic to $H$. 
We identified the bottom horizontal map with $\varepsilon^{\ast} \otimes id_Y : H^{\ast} \otimes_k Y \to k \otimes_k Y$ via the natural isomorphism $\Hom_k(H, Y) \cong H^{\ast} \otimes_k Y$, where $\varepsilon^{\ast} :  H^{\ast} \cong H \to k$ is corresponding to the counit $\varepsilon$ of $H$.

As the following diagram, we obtain the map $\Sigma^{-1}X \to P_f$ by the universality. 
\[ 
 \xymatrix@1{
 \Sigma^{-1}X \ar@{.>}[dr] \ar[dd] \ar[rr] & & 0 \ar[d]\\
  & P_f \ar[r] \ar[d]& X \ar[d]^f  \\
\Hom_k(H, Y) \ar[r]_{id} & \Hom_k(H, Y) \ar[r] & Y  
} \] 
Here, we write the pullback in the diagram as $\Sigma^{-1} X$  since the pullback is stably equivalent to desuspention in Lemma~\ref{lem23}. 

By virtue of the assumption that $\varepsilon (\lambda) \ne 0$, we have the commutative diagram
\[ 
 \xymatrix@1{
X \ar[drr]^{id_X} \ar[ddr]_{\lambda^{\ast} \otimes f} \ar@{.>}[dr] & & \\
  & P_f \ar[r] \ar[d]& X \ar[d]^f  \\
& \Hom_k(H, Y) \ar[r] & Y  
} \] 

where we identified the left slanting map $X \to \Hom_k (H , Y)$ with the map $\lambda^{\ast} \otimes f  : X \to H^{\ast} \otimes_k Y \cong \Hom_k(H, Y)$, where the element $\lambda^{\ast} \in H^{\ast} \cong H$ is corresponding to the left integral $\lambda$ of $H$. 
\end{defn} 

\begin{lemma}[Factorization]
For any maps $f : X \to Y$, we have a factorization $f : X \to E \to Y$ for some $E$, where $X \to E$ is a cofibration and $E \to Y$ is a trivial fibration. Similary, we have a factorization $f : X \to E' \to Y$, where $X \to E'$ is a trivial cofibration, and $E' \to Y$ is a fibration. 
\end{lemma}
\begin{proof}
In this case, we can take a cofibrant replacement constructed in \cite{Qi} and all objects are fibrant. Thus, it suffices to show that, if we have a cofibrant replacement and fibrant replacement, respectively, we obtain the desired factorization of this lemma. 

Our proof is based on the proof in the case of chain complexes which is given by Hovey and Christensen. 

Let us take a cofibrant replacement $QC_f \to C_f$ of $C_f$. Here, the map $QC_f \to C_f$ is a trivial fibration. Then, we have the following diagram : 
\[ 
 \xymatrix@1{
X \ar@{.>}[r] \ar[d]_{=} & P_u \ar@{.>}[d] \ar[r] & QC_f \ar[d]_{\simeq} \ar[r]^u & \Sigma X \ar[d]_{=} \\
X \ar[r]^f & Y \ar[r] & C_f \ar[r] & \Sigma X .
} \] 
Note that two vertial maps in the middle square is surjective and two horizontal maps in the middle square is $A$-split. We define the map $u$ as a map the right square commutes. We construct the map $P_u \to Y$ by using $A$-splittings and the argument before Lemma~\ref{equiv}. We have the map $\Sigma^{-1} \Sigma X \to P_u$ by the universality of desuspension. The map $P_u \to Y$ is surjective since it is composed of two surjections. Passing to the quotient triangulated category, the map $P_u \to Y$ is a weak equivalence. 

We show that the map $X \to P_u$ is a cofibration. 
We consider the following diagram
\[ 
 \xymatrix@1{
X  \ar[d] \ar[r] & M \ar[d] \\
P_u \ar[d] \ar[r] & N  \\
QC_f \ar@{.>}[uur] &  
} \] 
where $M \to N$ is a trivial fibration and $P_u \to QC_f$ is $A$-split. Since $QC_f$ is cofibrant, we have the dotted arrow. By using $A$-splitting endowed with trivial $H$-action, we have the lifting $P_u \to M$. 

Similarly, let us take a fibrant replacement $P_f \to RP_f$. Here, the map $P_f \to RP_f$ is a trivial cofibration. Then, we have the following diagram : 

\[ 
 \xymatrix@1{
\Sigma^{-1}Y \ar[r] \ar[d]_{=} & P_f \ar[d]_{\simeq} \ar[r] & X \ar@{.>}[d] \ar[r]^f & Y \ar[d]_{=} \\
\Sigma^{-1}Y \ar[r]^v & RP_f \ar[r] & C_v \ar@{.>}[r] & Y .
} \] 
Note that two horizontal maps in the middle square is $A$-split. The map $v$ is defined to be the left square commutes. We construct the map $X \to C_v$ by using $A$-splittings and the argument before Lemma~\ref{equiv}. We have the map $C_v \to \Sigma \Sigma^{-1}Y$ by the universality of suspension. The map $X \to C_v$ is a stable equivalence since it is composed of two stable equivalences. 

Note that $X \to C_v$ is also cofibration since $A$-split monomorphisms are included in the class of cofibrations. 

We show that the map $C_v \to Y$ is a fibration. 
We consider the following diagram
\[ 
 \xymatrix@1{
 & RP_f \ar[d] \\
M \ar[d] \ar[r] & C_v \ar[d] \\
N \ar[r]  \ar@{.>}[uur] &  Y
} \] 
where $M \to N$ is a trivial cofibration and $RP_f \to C_v$ is $A$-split. Since $RP_f$ is fibrant, we have the dotted arrow. By using $A$-splitting endowed with trivial $H$-action, we have the lifting $N \to C_v$. 
\end{proof}

Finally, we obtain the following proposition. 
\begin{prop}\label{main}
The category $\LMod_{A \# H}$ inherits a model structure with respect to the three classes of maps defined as in Definition~\ref{ModelAH}. 
\end{prop}
\begin{proof}
By the lemmas and the fact that all limits and colimits exists since it is a left module category. 
\end{proof}

\begin{cor}
The forgetful functor $U : \LMod_{A \# H} \to \LMod_H$ is the right Quillen functor. Here, we regard $\LMod_H$ as the model category as in Definition~\ref{ModelH}.  
\end{cor}
\qed

\begin{defn}[\cite{Qi}, Definition 6.1]\label{QiCof}
Let $M$ be a $A \# H$-module. 
$M$ is said to be cofibrant if for any surjective quasi-isomorphism $L \to N$ of left $A \# H$-modules, the induced map of $k$-vector spaces
\[
\Hom_{A \# H}(M, L) \to \Hom_{A \# H}(M, N)
\]
is surjective. 




\end{defn}

\begin{prop}
The cofibrant objects with respect to the model structure obtained as in Proposition~\ref{main} coinside with the cofibrant modules as in Definition~\ref{QiCof}  
\end{prop}
\begin{proof}
For any surjective weak equivalence of $B$-modules, Qi's cofibrant objects satisfies the lifting property. Therefore, it suffices to show that an $A \# H$-module map is a  fibration if and only if it is a surjection. By definition of a fibration, an $A \# H$-module map is a fibration if and only if it is a fibration of $H$-modules, which turns out to be a surjection~\cite[Lemma 2.2.6]{Hov}. Thus, Qi's cofibrant objects are cofibrant objects with respect to the relative model structure. 
\end{proof}

\begin{cor}
The homotopy category of the model structure on $\LMod_{A \# H}$ is the derived category $\mathcal{D}(A, H)$ of $A \# H$-modules as in \cite{Qi}. 
\end{cor}

We also mention that compact objects of left $H$-modules and left $A \# H$-modules, respectively, via the adjoint functor $A \otimes_k (-)$ and $U$. 

\begin{lemma}

\[
A \otimes_k (-) : \LMod_H \rightleftarrows \LMod_{A \# H} : U . 
\]
Here, $U$ is the functor which regards $A \# H$-modules as $H$-modules and $A \otimes_k (-)$ is the left adjoint, which is given by the free functor. 
\end{lemma}
\begin{proof}
If the right adjoint preserves colimits, the left adjoint  preserves compact objects. The forgetful functor $U$ preserves colimits.  
\end{proof}



\section{Cofibrant generation and some properties}

In this section, we explain the small object argument and show that the model structure defined in the previous section is cofibrantly generated.

\begin{defn}
Suppose $\mathcal{C}$ is a model category. We say that $\mathcal{C}$ is cofibrantly generated if there are sets $I$ and $J$ of maps such that : 
\begin{itemize}
\item The domains of the maps of $I$ are small relative to $I$-cell,
\item The domains of the maps of $J$ are small relative to $J$-cell, 
\item The class of fibrations is $J$-inj; i.e., they have the right lifting property with respect to every map in $J$, and 
\item The class of trivial fibrations is $I$-inj. 
\end{itemize}
We refer to $I$ as the set of generating cofibrations, and to $J$ as the set of generating trivial cofibrations. 
\end{defn}


\begin{prop}
The model structure on $\LMod_{A \# H}$ obtained in Proposition~\ref{main} is cofibrantly generated. Explicitely, the sets $I_{A \# H}=\{ A \otimes J \to A \otimes H \}$ and $J_{A \# H}=\{ 0 \to A \otimes H \}$ are generating cofibrations and generating trivial cofibrations, respectively, where $J$ runs over left ideals of $H$. 
\end{prop}
\begin{proof}
Note that $H$ is a Noetherian ring since it is finite dimensional over a field $k$ and since the forgetfull functor from the category of $A \# H$-modules to the category of $H$-modules is exact, it suffices to show that the class of fibrations is $J$-inj and the class of trivial fibrations is $I$-inj. 

We can see that maps satisfying right lifting property to those maps in $J$ and $I$, respectively, are just class of fibrations and class of trivial fibrations. 
Note that the forgetfull functor from the category of $A \# H$-modules to the category of $H$-modules is exact. 
\[
A \otimes_k (-) : \LMod_H \rightleftarrows \LMod_{A \# H} : U . 
\]
Here, $U$ is the functor which regards $A \# H$-modules as $H$-modules and $A \otimes_k (-)$ is the left adjoint, which is given by the free functor. 

Note that the class of fibrations and trivial fibrations are determined those classes in $\LMod_H$ since the two classes in $\LMod_{A \# H}$ are defined by using the stable model structure of $\LMod_H$ via the right adjoint $U : \LMod_{A \# H} \to \LMod_H$. 

First of all, a map satisfying right lifting property with respect to maps in $J$ if and only if a surjection as $A \# H$-module map by ~\cite{Hov} since a surjection of modules is just categorical epimorphism in this case.

Let us assume $H$ a Frobenius ring and if we take a class $I'= \{J \to H \}$ of $H$-module maps, where $J \subset H$ runs left ideals of $H$, the following two assertion which are follows from \cite{Hov}.  
 

($\blacklozenge$) An $H$-module map has the right lifting property with respect to maps in $I'$ if and only if the map is trivial fibration in $\LMod_H$. 

 

We will see that if a map of $A \# H$-modules $p : M \to N$ has the right lifting property with respect to $I$, then by regarding the map as an $H$-module map via the forgetfull functor $U$, we have the following combined diagram
\[ 
 \xymatrix@1{
J \ar[d]_i \ar[r] & A \otimes_k J \ar[d]_{id_A \otimes i} \ar[r] & M \ar[d]_p  \\
H\ar[r] & A \otimes_k H \ar[r] &  N ,
} \] 
so $p$ also has the right lifting prpperty with respect to $I'$. Together with ($\blacklozenge$), $p$ is a trivial fibration in $\LMod_{A \# H}$. Conversely, if a map $p: M \to N$ of $A \# H$-modules satisfying the right lifting property with respect to $I'$, we consider the right square of the diagram in $\LMod_{A \# H}$ and then assign the forgetfull functor $U$ and combine the left square of the diagram in $\LMod_H$. Then, we have a lifting $H \to M$ of $H$-modules. 
We set an $H$-module map $A \otimes_k H \to M$ by the map $H \to M$ and $A \otimes_k J \to M$ in the square for assign an element of $A$ to an element of $M$.  
\end{proof}

\bibliographystyle{amsplain} \ifx\undefined\bysame
\newcommand{\bysame}{\leavemode\hbox to3em{\hrulefill}\,} \fi

\end{document}